\documentclass[11pt]{amsart}
\usepackage{amsmath,amssymb,amsthm,color,graphicx,epsf}
\usepackage[margin=1in]{geometry}
\newtheorem{theorem}{Theorem}
\newtheorem{lemma}{Lemma}

\theoremstyle{definition}\newtheorem{definition}{Definition}
\newtheorem{example}{Example}
\newenvironment{theorem*}[2]{\vspace{8pt} \noindent \textbf{Theorem {#1}} ({#2})\textbf{.} \em}{}

\begin{document}

\title{A Cubical Antipodal Theorem}

\date{\today}

\author[Kinneberg, Mazel-Gee, Sondjaja, and Su]{
{\bf Kyle E. Kinneberg}\\
Department of Mathematics\\ Claremont McKenna College\\
850 Columbia Ave.\\ Claremont CA 91711, U.S.A.\\
{\tt \lowercase{kkinneberg09@cmc.edu}}\\
\\
{\bf Aaron Mazel-Gee}\\
Department of Mathematics\\Brown University\\
151 Thayer Street\\ Providence, RI 02912, U.S.A.\\
{\tt \lowercase{Aaron\_Mazel-Gee@brown.edu}}\\
\\
{\bf Tia Sondjaja}\\
School of ORIE\\ 206 Rhodes Hall\\ Cornell University\\
Ithaca, NY 14853, U.S.A.\\
{\tt \lowercase{ms999@cornell.edu}}\\
\\
{\bf Francis Edward Su$^*$}\\
Department of Mathematics\\ Harvey Mudd College\\
301 Platt Blvd.\\ Claremont CA 91711, U.S.A.\\
{\tt \lowercase{su@math.hmc.edu}}
}

\thanks{$^*$Corresponding author: voice 1-909-607-3616, fax 1-909-621-8366}
\thanks{The authors gratefully acknowlege partial support by NSF Grant DMS-0701308 (Su), NSF-REU Grant DMS-0453284 (Claremont Colleges REU Site), and a Beckman Research Grant at Harvey Mudd College.}

\begin{abstract}
The classical Lusternik-Schnirelman-Borsuk theorem states that
if a $d$-sphere is covered by $d+1$ closed sets, then at least one of
the sets must contain a pair of antipodal points.  In this paper, we
prove a combinatorial version of this theorem for hypercubes.  It is not hard to show that for any
cover of the facets of a $d$-cube by $d$ sets of facets,
at least one such set contains a pair of antipodal ridges.  However, we show that for any
cover of the ridges of a $d$-cube by $d$ sets of ridges, at least one set 
must contain a pair of antipodal $k$-faces, and we determine the maximum
$k$ for which this must occur, for all dimensions except $d=5$.
\end{abstract}

\maketitle

\section{Introduction}

Let $S^d$ denote the standard $d$-sphere, the set of all points of
distance one from the origin in ${\mathbb R^{d+1}}$.  This is a
$d$-dimensional geometric object with a natural notion of {\em
antipode}: if $x$ is in $S^d$, so is $-x$, the {\em antipode} of $x$.
A classical result of Lusternik-Schnirelman-Borsuk \cite{Bors33,
LuSc30} (which we refer to as the LSB theorem)
says that if $S^d$ is covered by $d+1$ closed sets, then at
least one of the sets contains a pair of antipodes.

The motivation for this paper grew out of a desire to develop a combinatorial version of the LSB theorem that might provide an alternate combinatorial route to proving this classical topological result.  Such methods have been used, for instance, to prove the Brouwer fixed point theorem and the Borsuk-Ulam theorem via combinatorial lemmas due to Sperner and Tucker, e.g., see \cite{KKM29,Sper28,Tuck45}.

Let $C^d$ denote the $d$-cube, sometimes called the $d$-dimensional
{\em hypercube}, which is a polytope in ${\mathbb R^{d}}$ that is the
product of $d$ line segments.  (For convenience, we take each line
segment to be $[0,1]$.)  This polytope $C^d$ has faces which are also
cubes; these are of dimension $d$ or less.  For simplicity, 
we call a $k$-dimensional face a $k$-face.
A combinatorial version of $S^{d-1}$ is the boundary of $C^d$, since
this boundary is topologically a $(d-1)$-sphere.  Note that this
boundary is the union of all {\em facets}, which are $(d-1)$-faces of
one less dimension than the cube itself.  

We may ask whether there is
a combinatorial version of the LSB theorem in this context.  
Such a
version might begin by covering the facets of $C^d$ by $d$ sets that
are unions of facets.
Let us call such a collection of sets a {\em facet cover}.  The
classical LSB theorem would then guarantee the existence of a set in
the facet cover that contains a pair of antipodal points, and by
combinatorial methods we shall see in Theorem \ref{facet-cover} that,
in fact, we can guarantee the existence of a set that contains a pair
of antipodal $(d-2)$-faces.

What may be surprising is that we can say something about the
existence (and dimension) of antipodal faces contained in some 
{\em ridge cover}, which is a collection of $d$ 
sets that are unions of $(d-2)$-faces of $C^d$.  
For instance, consider the $3$-cube.  If we cover its edges with
three sets $A_1, A_2, A_3$ that are unions of whole edges, we find that
some $A_i$ must contain a pair of antipodal vertices.  However, the classical
LSB theorem does not say anything about the existence of such a pair.

Before describing our main results, we give some terminology and background.

\subsection*{Terminology and Notation}

For $d \in \mathbb{N}$, recall that $C^d$ is the convex hull of the
$2^d$ points in $\mathbb{R}^d$ whose coordinates are all either {\tt0}
or {\tt1}.  For example, the 1-cube is the line segment with endpoints
{\tt0} and {\tt1}, the 2-cube is the square with vertices {\tt00},
{\tt01}, {\tt10}, and {\tt11}, etc.  Thus, any vertex of $C^d$ is
denoted by a $d$-tuple of {\tt0}'s and {\tt1}'s.  More generally,
observe that a (closed) $k$-face may be specified by a choice of $d-k$
coordinate positions whose values are fixed at either {\tt0} or {\tt1}.  The 
$k$ remaining coordinates vary throughout this $k$-face; we call them {\em varying} coordinate positions of the $k$-face.

Thus every $k$-face of $C^d$ may be represented by a $d$-tuple of {\tt X}'s, {\tt 0}'s, and {\tt 1}'s, by putting an {\tt X} in the $k$ coordinate positions that are {\em varying} and putting the fixed values in the coordinate positions that are {\em fixed}.  This is called the {\em coordinate representation} of a face.
For example, the 2-face of $C^3$ containing the vertices
{\tt001}, {\tt101}, {\tt111}, and {\tt011} is represented by {\tt XX1}.  
We say this face has a {\em fixed coordinate} in the 3rd {\em position} with {\em coordinate value} $1$.  We say two faces {\em differ} in the $j$-th coordinate if their coordinate representations have different symbols in the $j$-th position; else they {\em agree} in that coordinate.  For example the faces {\tt XX01} and {\tt X0X1} differ in the 2nd coordinate and in the 3rd coordinate, and the two faces agree in the 1st coordinate and in the 4th coordinate.

We now define the concept of antipodes for faces of a $d$-cube.
A vertex $v_1$ is the \textit{antipode} of $v_2$ if $v_1$ and $v_2$
differ in every coordinate.  Clearly this relationship is symmetric,
and we may sometimes say $v_1$ and $v_2$ are \textit{antipodal} to
each other.  Similarly, we say a $k$-face $F_1$ is the \textit{antipode}
of a $k$-face $F_2$ if the antipode of each vertex of $F_1$ lies in $F_2$ 
and vice versa; we may also say that $F_1$ and $F_2$ are
\textit{antipodal}.  In coordinate representations, $F_1$ and $F_2$ are antipodal if and only if they have exactly the same coordinate positions fixed but differ in those fixed coordinate positions.

Furthermore, we say that two
sets $A,B$ are \textit{$k$-antipodal} to each other 
if $A$ contains a $k$-face whose antipode is in $B$.
(Clearly if $A$ is $k$-antipodal to $B$, then $A$ is 
$k'$-antipodal to $B$ for all $k' \leq k$.)  
Finally, a set that is $k$-antipodal to itself is
\textit{$k$-self-antipodal}.  We take the convention that the empty
set is an antipode of itself, which is vacuously true.

Since this paper primarily deals with faces of codimensions 1 through
4, we use the following names.  The term \textit{facet} denotes a
$k$-face of codimension 1, \textit{ridge} denotes a $k$-face of
codimension 2, \textit{peak} denotes a $k$-face of codimension 3, and
we introduce the term \textit{pinnacle} to denote a $k$-face of
codimension 4.  For example, on $C^4$, facets are 3-dimensional cells,
ridges are 2-dimensional faces, peaks are edges, and pinnacles are
vertices.  We consider the empty set to be a face of dimension $-1$.

As the $d$-cube is symmetric, we choose a specific orientation without
loss of generality.  Accordingly, we often refer to {\tt X...X0} as
the \textit{bottom facet} and to {\tt X...X1} as the \textit{top facet}. 
Furthermore, we refer to a $k$-face which intersects
both the top and bottom facets as a \textit{spanning $k$-face}.

We note two combinatorial facts about the $d$-cube which can be proven
without difficulty: it contains $2^{d-k}{d \choose k}$ $k$-faces, and
$d$ edges meet at any of its vertices.
We are now ready to introduce a key definition.

\begin{definition}
An $n$-set \textit{ridge cover of} $C^d$ is a collection of sets $A_1,
\ldots, A_n$, each a union of ridges, 
such that every ridge of $C^d$ is in at least one set.
\end{definition}

In this paper, most ridge covers we consider will have exactly $d$
sets, where $d$ is the dimension of the cube $C^d$.

\section{A cubical LSB Theorem for ridge covers}

The LSB theorem guarantees that in a collection of $d$ closed sets
covering the facets of $C^d$ (the topological equivalent of
$S^{d-1}$), at least one of the sets contains a pair of antipodal points.  
Our main theorem gives a corresponding result for ridge covers of $C^d$.

\begin{theorem}[Cubical LSB]
\label{complete}
In any $d$-set ridge cover of $C^d$ there must be a set containing a pair of 
antipodal $k(d)$-dimensional faces, where 
$$
k(d)=
\left\{
\begin{array}{ll}
d-2, & d=1 \\
d-3, & 2 \leq d \leq 4 \\  
d-4, & d \geq 5. \\
\end{array}
\right.
$$
Also, with the possible exception of $d=5$,
$k(d)$ is sharp in the
sense that there does not have to be a set in the cover that contains
a pair of antipodal $(k(d)+1)$-faces.  
\end{theorem}

As mentioned earlier, this theorem implies that in a $3$-set ridge cover of
$C^3$ at least one set will contain antipodal vertices (peaks).  It is
possible, however, for none of the sets to contain antipodal edges (ridges).
This theorem also says that in a $4$-set ridge cover of $C^4$ 
at least one set will contain antipodal edges (peaks).
Similarly, it is not necessarily true that one of the sets will
contain antipodal 2-faces (ridges).  (See Section \ref{sec:sharpness}.)

To prove the theorem, we break it up into a series of intermediate results which together imply Theorem \ref{complete} for $d \geq 3$.  (Theorem \ref{complete} is clearly true when $d=1,2$.)

\begin{theorem}[Weak cubical LSB]
\label{weak}
For $d \geq 3$, any $d$-set ridge cover of $C^d$ must have at least one set that contains
a pair of antipodal vertices.
\end{theorem}

\begin{theorem}[Medium cubical LSB]
\label{medium}
For $d \geq 3$, any $d$-set ridge cover of $C^d$ must have at least one set that contains
a pair of antipodal $(d-4)$-faces.
\end{theorem}

\begin{theorem}[Strong cubical LSB]
\label{strong}
For $2 \leq d \leq 4$, any $d$-set 
ridge cover of $C^d$ must have at least one set that contains a
pair of antipodal $(d-3)$-faces.
\end{theorem}

\begin{theorem}[The sharpness of $k(d)$]
\label{sharpness}
For all $d\neq 5$, there exists a $d$-set ridge cover of $C^d$ in which no
set contains a pair of antipodal $(k(d)+1)$-faces.  When $d=5$, there does not have to be a set in the cover that contains a pair of antipodal $(d-2)$-faces.
\end{theorem}

We prove each of these intermediate results in its own section. As an
introduction to our methods, we first prove the simple statement about
facet covers mentioned earlier.

\section{A cubical LSB theorem for facet covers}

For a facet cover of $C^d$, the classical LSB theorem guarantees the existence of a set in the cover that contains a pair of antipodal points on the facets of $C^d$.  We strengthen this result with the following theorem.

\begin{theorem} \label{facet-cover}
For all $d \geq 1$, a $d$-set facet cover of $C^d$ has a set that contains a
pair of antipodal $(d-2)$-faces, i.e., antipodal ridges.  The dimension $(d-2)$ is sharp; antipodal faces of larger dimension cannot be guaranteed.
\end{theorem}

\begin{proof}
Recall that $C^d$ has $2d$ facets, which are covered by the $d$ sets of the facet cover.  Therefore, some set $A_i$ contains at least two facets.

Without loss of generality, suppose the facet ${\tt X...X0}$ is contained in $A_i$. Thus, $A_i$ covers all ridges of ${\tt X...X0}$, and to avoid containing antipodal ridges, it must not cover any ridges of ${\tt X...X1}$. Clearly, then, ${\tt X...X1}$ cannot lie in $A_i$, which implies that any other facet in the set must have its fixed coordinate among the first $d-1$ coordinates.  Hence the intersection of this facet and ${\tt X...X1}$ is non-empty and must be a ridge.  Then $A_i$ contains ${\tt X...X0}$ and a ridge in ${\tt X...X1}$, which must be the antipode of a ridge in ${\tt X...X0}$.  Therefore $A_i$ contains a pair of antipodal ridges.

For the sharpness of $(d-2)$, there is a $d$-set facet cover of $C^d$ that does not contain a pair of antipodal facets.  In particular, consider the facet cover in which $A_i$ contains two facets, one with the $i$-th coordinate fixed at ${\tt 0}$ and the other with the $(i+1)$-th coordinate fixed at ${\tt 1}$ (if $i=d$, then interpret the $(i+1)$-st as the $1$-st coordinate).  Clearly every facet is in some set of this cover, but no set of this cover contains antipodal facets.
\end{proof}

We now turn our attention to ridge covers.

\section{The weak cubical LSB theorem}

We prove the weak cubical LSB theorem for the case of the 3-cube, and then use the framework from that proof to argue the general case.

\begin{figure}[htb]
\begin{center}
\includegraphics[height=2in]{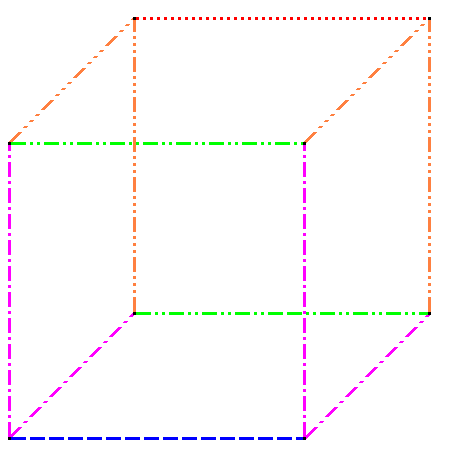}
\end{center}
\caption{$C^3$ with edges ``colored'' according to the following
  legend: 
red=dotted, 
orange=dash-dot-dot-dot,
green=dash-dot-dot,
purple=dash-dot, 
blue=dashed.
}
\label{3cubeweak}
\end{figure}

\begin{lemma}[Weak cubical LSB theorem on $C^3$]
\label{weakC3}
Any $3$-set ridge cover of $C^3$ must have at least one set that contains a pair of
antipodal vertices.
\end{lemma}

\begin{proof}
Note that on $C^3$, a ridge is an edge. We prove the statement by showing that if no cover sets contain a pair of 0-antipodal edges, then each cover set must be the boundary of a 2-face.  Then, we show that three sets of this kind cannot cover the edges of the 3-cube.

Since $C^3$ contains twelve edges, there must be a cover set which contains at least four edges.  Let this set be $A_i$.  Assume, without loss of generality, that {\tt X00} (the blue edge in Figure \ref{3cubeweak}) is in $A_i$.  This immediately eliminates the red edge {\tt X11} and the orange edges {\tt0X1}, {\tt1X1}, {\tt01X}, and {\tt11X} as potential members of $A_i$, since they are all 0-antipodal to the blue edge.  (The red edge is actually 1-antipodal to the blue edge.)  This leaves us the four purple edges {\tt0X0}, {\tt00X}, {\tt1X0}, and {\tt10X}, each of which shares a vertex with the blue edge, and the two green edges {\tt X01} and {\tt X10}, which are both parallel (and not at all antipodal) to the blue edge.

Suppose that three of the remaining edges of $A_i$ were purple.  But there are two pairs of purple edges which are 0-antipodal ({\tt00X} and {\tt1X0} form one, {\tt0X0} and {\tt10X} form the other).  The pigeonhole principle guarantees that at least one of the pairs will have both members chosen.  This contradicts the condition that $A_i$ not be 0-self-antipodal.

Therefore, at least one of the edges in $A_i$ is a green edge.  In fact, it can only be one green edge since the two green edges are antipodal; thus, choosing one necessarily excludes the other.  Further, once we have chosen a green edge, the two purple edges which do not meet it are also excluded, since they must instead meet the other green edge, and the two green edges are 1-antipodal.

This implies that $A_i$ must contain four edges which bound a face.  If it were to contain any more edges, it would be 0-self-antipodal (as we have shown by elimination), so it must contain exactly these four edges.

Since $A_i$ was arbitrary (up to the fact that it contained four edges), it is clear that \textit{each} set $A_1,A_2,A_3$ must contain exactly four edges, and each must be the boundary of a face.  Hence no two sets can cover the same edge; otherwise some other edge must not be covered. In particular, $A_1$ and $A_2$ are disjoint and therefore must bound parallel faces of $C^3$. The boundary of any other face must share an edge with each of $A_1$ and $A_2$, contradicting the fact that $A_3$ must be disjoint from $A_1$ and $A_2$. Thus, it is impossible to cover the edges (ridges) of $C^3$ with three cover sets such that no set is 0-self-antipodal.
\end{proof}

In this proof, we established a fact that will be useful later, so we record it here:

\begin{lemma} \label{C3edge}
Let $A$ be a collection of at least four edges in $C^3$ with no antipodal vertices. Then $A$ must bound a 2-face of $C^3$.
\end{lemma}

We now move on to the case for $C^d$, which we prove in much the same way as we did the weak cubical LSB theorem on $C^3$. In particular, the colorings of ridges are analogous.

\begin{theorem*}{\ref{weak}}{The weak cubical LSB theorem on $C^d$}
For $d \geq 3$, any $d$-set ridge cover of $C^d$ must have a set that contains
a pair of antipodal vertices.
\end{theorem*}

\begin{proof}
Again, we prove the statement by showing that if no cover sets contain a pair of antipodal points, then each cover set must be the boundary of a facet.  Then, we show that $d$ sets of this kind cannot cover the ridges of $C^d$.

We know that on average, the $d$ sets of ridges covering the
$2^{d-(d-2)}{d \choose {d-2}}=2d(d-1)$ ridges must contain at least
$2(d-1)$ ridges each.  Thus, there must be a set containing at least
$2(d-1)$ ridges.  Let this set be $A_i$.  Assume, without loss of
generality, that $A_i$ contains the ridge {\tt X...X00}.  We call this
the blue ridge.  Then $A_i$ cannot contain the ridge {\tt X...X11}
(the red ridge, which is antipodal to the blue ridge), as well as any
ridge with which it intersects (the orange ridges).

Fortunately, our notation affords us an easy way to determine which ridges are 0-antipodal to the blue ridge.  Antipodal vertices differ in every coordinate.  Thus, if a ridge is to be 0-antipodal to the blue ridge, it must contain a vertex with a {\tt1} in both of the last two coordinates; this ridge is either the unique red ridge or one of the orange ridges.  Therefore, a ridge will be antipodal to the blue ridge if it is of the form {\tt...XX}, {\tt...X1}, {\tt...1X}, or {\tt...11}.  (The first three types determine orange ridges, and the last type determines the red ridge.)

Two ridges are parallel if they vary in all the same coordinates, and keep the same coordinates constant.  Then there are three ridges parallel to the blue ridge (besides itself): {\tt X...X01}, {\tt X...X10}, and {\tt X...X11}.  The last is the red ridge, but the others are not at all antipodal to the blue ridge; these, then, are the green ridges.  Note that there are always exactly two of them.

The ridges that are not 0-antipodal to the blue ridge share vertices with the blue ridge and with a green ridge.  This is because they must be of the form {\tt...X0} or {\tt...0X} (since all the other possible ending patterns have already been classified).  These are the purple ridges.  There are $4(d-2)=4d-8$ of them.

We are now equipped to show that $A_i$ must be the boundary of a
facet.  Recall that $A_i$ contains the blue ridge and at least $2d-3$
other ridges. In a completely analogous fashion to the previous proof,
suppose first that at least $2d-3$ of these other ridges were purple.
But the purple ridges come in 0-antipodal pairs.  In general, given a
purple ridge, a 0-antipodal purple ridge can be found by switching the
last two coordinates and reversing the value of the other fixed
coordinate.  For example, the ridge {\tt0X...X0} is paired with
{\tt1X...X0X}.  Again by the pigeonhole principle, since there are
$2d-4$ such pairs of purple ridges and we are attempting to choose
$2d-3$ purple ridges, at least one pair must have both members chosen.
This contradicts the condition that $A_i$ not be 0-self-antipodal.

Therefore, at least one of the ridges in $A_i$ is a green ridge.
Again, it can only be one green ridge, since the two green ridges are
antipodal, so that choosing one necessarily excludes the other.
Further, once we have chosen a green ridge, any purple ridge meeting
the other green ridge is also excluded, as the green ridges are
antipodal to each other.  (Any purple ridge meets exactly one of the green
ridges; a purple ridge of the form {\tt...X0} will meet the green
ridge {\tt X...X10}, while a purple ridge of the form {\tt...0X} will
meet the green ridge {\tt X...X01}.)

Now $A_i$ may only contain the blue ridge, one green ridge and the half of the purple ridges which meet this green ridge.  These determine the boundary of a facet: for example, if we choose the green ridge {\tt X...X10}, we would be left with the purple ridges of the form {\tt...0X}, so that the facet spanned would be {\tt X...X0}.  Note that $A_i$ cannot contain any more ridges, else $A_i$ would be 0-self-antipodal (as we have shown by elimination).

Since $A_i$ was arbitrary (up to the fact that it contained at least $2(d-1)$ ridges), it is clear that \textit{each} set $A_1,...,A_d$ must contain exactly $2(d-1)$ ridges, and each must bound a facet. Hence no two sets can cover the same ridge; otherwise some other ridge must not be covered. In particular, $A_1$ and $A_2$ are disjoint and therefore must bound parallel facets of $C^d$. Since $d \geq 3$, there exists another cover set. But the boundary of any other facet must share a ridge with each of $A_1$ and $A_2$, contradicting the fact that all cover sets must be disjoint. Thus, it is impossible to cover the ridges of $C^d$ with $d$ cover sets such that no set is 0-self-antipodal.
\end{proof}

\section{The medium cubical LSB theorem}

In order to prove the medium cubical LSB theorem, we combine the 
weak cubical LSB theorem with the following lemma.

\begin{lemma}
\label{pinnacleintersection}
For $d \geq 4$, any two ridges of $C^d$ that intersect each other
contain a pinnacle in their intersection.
\end{lemma}

\begin{proof}
A ridge of $C^d$ is a $(d-2)$-cube, which has $d-2$ edges meeting at
any vertex.  If two ridges of $C^d$ intersect, they share at least a
vertex. But there are only $d$ edges that meet at any vertex of $C^d$,
so the ridges must share at least $d-4$ edges. This implies that they
share a pinnacle spanned by these edges.
\end{proof}

The medium cubical LSB theorem now follows readily.

\begin{theorem*}{\ref{medium}}{Medium cubical LSB theorem}
For $d \geq 3$, any $d$-set ridge cover of $C^d$ must have at least one set that contains
a pair of antipodal $(d-4)$-faces.
\end{theorem*}

\begin{proof}
The medium cubical LSB theorem is trivial on $C^3$, and is equivalent to the weak cubical LSB theorem on $C^4$.  For $d \geq 5$, we parallel the proof of the weak cubical LSB theorem by attempting to construct a ridge cover such that no cover set contains antipodal pinnacles, and obtain a contradiction.

In that proof, we showed that if the blue ridge is covered by some set $A_i$ that contains at least $2(d-1)$ ridges, the red ridge and orange ridges could not be contained in $A_i$.  The same is true here; the red ridge is antipodal to the blue ridge, and by Lemma \ref{pinnacleintersection}, every orange ridge must actually share a pinnacle with the red ridge and must therefore be $(d-4)$-antipodal to the blue ridge.

What remain to be considered are the 2 green ridges and the $4d-8$ purple ridges.  Each of the $2d-4$ pairs of purple ridges constructed in the previous proof is actually $(d-3)$-antipodal, and therefore at most one of each pair can be in $A_i$.  Since $A_i$ contains at least $2d-2$ ridges and one is the blue ridge, then as before, by the pigeonhole principle, the $2d-3$ remaining ridges of $A_i$ cannot all be purple.  Thus, there must be a green ridge, and since the two green ridges are $(d-2)$-antipodal, there must be exactly one green ridge in $A_i$.  The purple ridges each share a peak with a green ridge, so they are $(d-3)$-antipodal to the other green ridge.  This means that once a green ridge has been chosen, all the purple ridges meeting the other green ridge cannot belong to $A_i$. Therefore, $A_i$ contains exactly $2d-2$ ridges, and they bound a facet.

By the same reasoning as in the proof of Theorem \ref{weak}, each cover set must bound a facet and $d$ such sets cannot cover all the ridges of $C^d$, a contradiction.
\end{proof}

\section{The strong cubical LSB theorem}

Although the medium cubical LSB theorem (Theorem \ref{medium})
holds for all $d \geq 3$, we can strengthen our result for $d=4$.

\begin{theorem*}{\ref{strong}}{Strong cubical LSB theorem}
For $2 \leq d \leq 4$, any $d$-set ridge cover of $C^d$ must have at least one 
set that contains a pair of antipodal $(d-3)$-faces.
\end{theorem*}

\medskip

For $d=2$, a peak is simply the empty set, so
Theorem \ref{strong} is trivial. For $d=3$, Theorem \ref{strong}
reduces to Theorem \ref{weak}, proved already. For $d=4$, however,
Theorem \ref{strong} says that a $4$-set ridge cover of $C^4$ has a
set containing a pair of antipodal edges.
This is a stronger statement than Theorem \ref{medium}, which
only guarantees a set containing a pair of antipodal vertices.

Recall some terminology. We refer to the facet {\tt XXX0} as the bottom  facet and the facet {\tt XXX1} as the top facet. We also refer to the ridges which intersect both the top and bottom facets as \textit{spanning ridges}. The spanning ridges are those with an {\tt X} in the last coordinate. 
We begin with a few lemmas.

\begin{lemma}
\label{2ridgesnopeaks}
If two ridges in $C^d$ do not contain a pair of antipodal peaks, then there must be a coordinate position that both ridges fix at the same coordinate value, or the two coordinate positions that each ridge fixes must be disjoint.
\end{lemma}

For example,  in $C^5$ the ridge ${\tt XXX10}$ in $C^5$ has fixed 4th and 5th coordinates.  The ridges ${\tt XXX10}, {\tt XX0X1}$ contain the antipodal pair of peaks 
${\tt XX110}$ and ${\tt XX001}$).  The ridges ${\tt XXX10}$ and ${\tt XX1X0}$ do not contain an antipodal pair of peaks because they both fix the 5th coordinate at ${\tt 0}$. 
The ridges ${\tt XXX10}$ and ${\tt X11XX}$ do not contain a pair of antipodal peaks, and their fixed coordinates occupy different positions.
 
\begin{proof}
Every ridge has two fixed coordinates.  For two ridges $r_1, r_2$ to contain an antipodal pair of peaks, $r_1$ and $r_2$ must have $d-3$ of their varying coordinate positions in common, and in all other coordinate positions, $r_1$ and $r_2$ must differ.  Thus if $r_1, r_2$ do not contain an antipodal pair of peaks, they either have fewer than $d-3$ varying coordinate positions in common (which occurs when their fixed coordinate positions are disjoint) or they must both fix a coordinate position at the same value.
\end{proof}

\begin{lemma}
\label{4spanningridges}
If a set of ridges of $C^4$ does not contain a pair of antipodal edges,
then it must contain four or fewer spanning ridges.
\end{lemma}

\begin{proof}
Suppose that a set $A_i$ contains five spanning ridges; we shall obtain
a contradiction.  Each of these ridges intersects the bottom facet at an edge,
 so $A_i$ contains five edges from the bottom facet (here, a $3$-cube).
One can verify that $A_i$ must therefore 
contain at least five vertices of the bottom facet that also lie on spanning ridges.

Any facet in $C^4$ contains eight vertices and therefore contains four
pairs of vertices which are antipodal pairs relative to that facet. Thus, $A_i$
must contain at least one pair of vertices $v_1, v_2$ that are
antipodes with respect to the bottom facet and also lie on spanning ridges.  For simplicity,
define $\hat{v}_j$, the \textit{corresponding vertex of} $v_j$, to be
the vertex whose coordinate notation differs from that of $v_j$ in
only the last coordinate. For example, {\tt0010}'s corresponding
vertex is {\tt0011}. (Note that {\tt0011} has the same ``position'' in the
top facet as {\tt0010} has in the bottom facet.)

Since $v_1$ and $v_2$ lie on spanning ridges contained in $A_i$, then
the edge with endpoints $v_1$ and $\hat{v}_1$ and the edge with
endpoints $v_2$ and $\hat{v}_2$ are contained in $A$. But these two
edges are antipodal, contradicting our hypothesis.
\end{proof}

\begin{lemma}
\label{3ridgesnopeaks}
If three ridges of $C^d$ do not contain any antipodal peaks and each have fixed $j$-th coordinate, then their $j$-th coordinates must agree.  In other words, they must all lie on the same facet.
\end{lemma}

\begin{proof}
If not, suppose without loss of generality that ridge $r_1$ fixes the $j$-coordinate at ${\tt 0}$, and ridge $r_2, r_3$ fix the $j$-th coordinate at ${\tt 1}$. Since $r_1$ and $r_2$ both fix the $j$-th coordinate but not at the same value, by Lemma \ref{2ridgesnopeaks}, they must both fix another coordinate, say the $k$-th, and agree in that coordinate.  Since they are ridges, $r_1$ fix only two coordinates, the $j$-th and $k$-th, and they must differ in the $j$-th and agree in the $k$-th coordinates.  Similarly, $r_1$ and $r_3$ must fix the $j$-th and $k$-th coordinates, differ in the $j$-th and agree in the $k$-th.  Therefore ridges $r_2$ and $r_3$ must have identical coordinate representations and be identical facets, a contradiction.
\end{proof}

The following lemma will show that the kind of sets that have no
antipodal edges and appear in a $4$-set ridge cover of $C^4$
must be rather special.

Given a vertex $v$, let $Ast(v)$ denote the set of all ridges that
contain $v$.  We call the set an {\em asterisk set} and call $v$ 
the {\em asterisk vertex} for that set.  Given a coordinate
representation of $v$ in $C^4$, $Ast(v)$ contains the ridges denoted by 
replacing two of the four fixed coordinates of $v$ by {\tt X}'s.
Thus, there are ${4 \choose 2} = 6$
ridges in $Ast(v)$. In general, note that for a given $v \in C^d$,
$Ast(v)$ contains the ridges formed by letting $d-2$ of the $d$
coordinates vary. Therefore, $Ast(v)$ contains ${d \choose 2}$
ridges.  By construction, an asterisk set cannot contain antipodal
edges (since ridges that share an {\tt X} in one coordinate must also
share a fixed coordinate as well).

\begin{lemma}
\label{6ridgesperset}
If no set in a $4$-set ridge cover of $C^4$ contains antipodal edges,
each set must contain exactly six ridges. Furthermore, any six-ridge
set must either contain all of the ridges in a single facet of $C^4$
or contain all of the ridges of an asterisk set.
\end{lemma}

\begin{proof}
A $4$-set ridge cover of $C^4$ must cover $24$ ridges, hence there must be a set $A_i$ of the cover that contains at least $6$ ridges.
Since each of those $6$ ridges is in exactly $2$ facets (corresponding to exactly $2$ fixed coordinates), then by the pigeonhole principle (noting $2$ times $6$ ridges correspond to $12$ fixed coordinates among $4$ possible coordinates) we see that one of the $4$ coordinates, say the $j$-th, must be fixed by at least $3$ ridges of $A_i$.  

By Lemma \ref{3ridgesnopeaks}, these $3$ ridges all lie on the same facet; without loss of generality, suppose this facet is the top facet ${\tt XXX1}$.

We first note that the bottom facet cannot contain any ridges of $A_i$, for if it did, the ridge on the bottom facet together with two ridges in the top facet would constitute a set of three ridges with same fixed coordinate position, but differing coordinate values in that position, contradicting Lemma \ref{3ridgesnopeaks}.
Then one of the following must hold:
\begin{enumerate}
\item $A_i$ contains 3 ridges of the top facet, no ridges in the bottom facet
\item $A_i$ contains 4 or 5 ridges of the top facet, no ridges in the bottom facet
\item $A_i$ contains the 6 ridges of the top facet.
\end{enumerate}
We consider each of these cases in turn. 

\begin{enumerate}
\item
Suppose that $A_i$ contains at least 3 ridges of the top facet and no ridges of the bottom facet.  It must therefore have at least 3 spanning ridges.  There are two general ways that $A_i$ may contain three ridges in the top facet. One way is depicted in Figure \ref{4cube3mainast}, in which the three ridges intersect at a vertex. 

\begin{figure}[htb]
\begin{center}
\scalebox{0.4}{\includegraphics{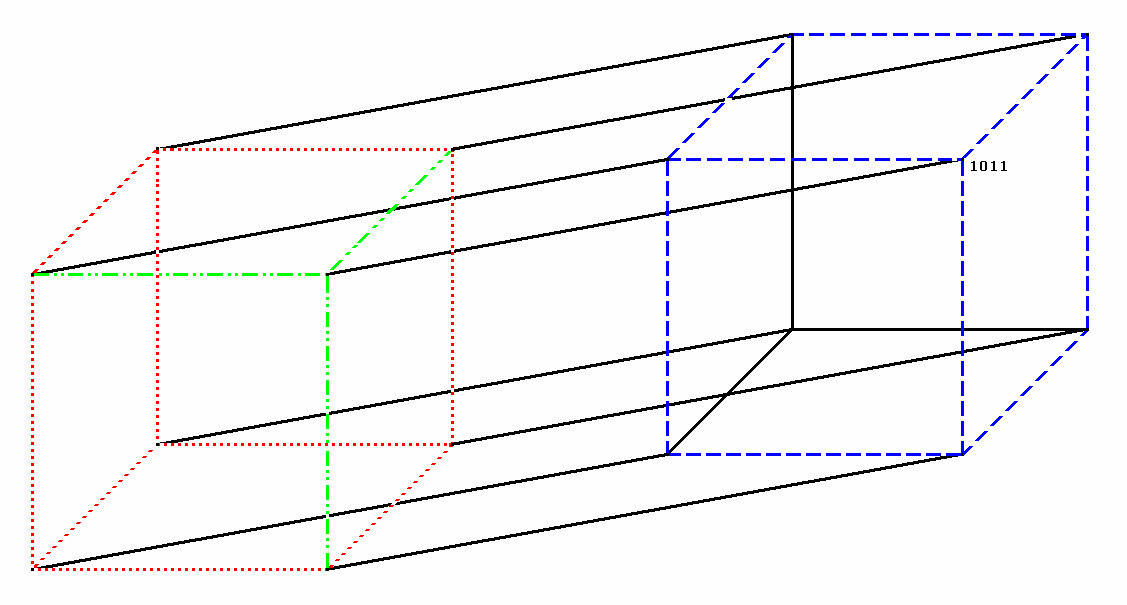}}
\end{center}
\caption{$C^4$ in situation (3) of
Lemma \ref{6ridgesperset}, where $A_i$ contains three ridges in
the top facet (here, at right) that meet at a vertex.  
Note that ridges of $C^4$ are $2$-faces, and
the three ridges are outlined with blue edges.
The color legend is as follows: blue=dashed, green=dash-dot-dot,
red=dotted.
}
\label{4cube3mainast}
\end{figure}

Without loss of generality, let $A_i$ contain the ridges outlined in
blue ({\tt X0X1}, {\tt XX11}, and {\tt 1XX1}). Then, the edges of
these ridges have antipodes ({\tt X000}, {\tt X100}, {\tt X110}, {\tt
  00X0}, {\tt 01X0}, {\tt 11X0}, {\tt 0X00}, {\tt 0X10}, {\tt 1X00}),
indicated in red. The spanning ridges that intersect the bottom main
facet at these edges cannot lie in $A_i$. Thus, the only spanning
ridges that can lie in $A_i$ are {\tt 10XX}, {\tt X01X}, and {\tt
  1X1X}. (The intersections of these ridges with the bottom facet
are outlined in green). Notice, however, that the blue and the green
ridges all contain the vertex {\tt1011}, and hence
$A_i = Ast({\tt1011})$.   See Figure \ref{4cube3mainast}.

The other way that $A_i$ may contain three ridges in the top facet
is if they do not have a common intersection;
they form a ``U'' shape with two of the ridges intersecting the third
ridge at different edges. This, however, results in $A_i$ containing
ten edges in the top facet. Therefore, there are only two edges
in the bottom facet that are not antipodes to these; but by 
assumption, $A_i$ contains three spanning ridges, and
these intersect the bottom facet at three distinct edges.  Thus one of these edges of $A_i$ in the bottom facet is antipodal to an edge of $A_i$ in the top facet, contradicting our assumption that $A_i$ contains no antipodal edges.  Thus the ``U'' shape cannot occur.

\item
Suppose that $A_i$ contains 4 or 5 ridges of the top facet, and no ridges belonging to the bottom facet.  For any four (respectively, five)
ridges in the top facet, at least eleven (resp. twelve) of that
facet's edges will be covered by $A_i$. Hence at most one (resp. no)
edge in the bottom facet can be covered by the two (resp. one)
spanning ridges, a contradiction.  Thus this case cannot occur.

\item
Finally, suppose that $A_i$ contains the six ridges of the top facet. This set does not contain antipodal edges, so it is a valid construction of a six-ridge set.
\end{enumerate}

Hence, we see there are only two ways to construct a set with at least 6 ridges so that it does not contain a pair of antipodal edges: the set must contain all six ridges in a single facet, or the set must contain all six ridges in an asterisk set.  We now show that such sets cannot contain more than 6 ridges.

Suppose that $A_i$ contains the ridges of a single facet $F$. If there were any other ridge in $A_i$, it would intersect the antipode of an edge in $F$.  Alternately, suppose that $A_i$ contains the ridges of an asterisk set, which includes sixteen edges. The $4$-cube contains 32 edges, in sixteen pairs of antipodal edges. If there were any other ridge in $A_i$, it would contain at least seventeen edges, which by the pigeon-hole principle would include a pair of antipodal edges. Thus, this set also may not contain more than six ridges.

Therefore, in order for the four sets of a ridge cover to cover all 24 ridges in $C^4$, each must contain exactly six ridges.
\end{proof}

\begin{lemma} \label{nofacet4}
Let $A_i$ be one of four sets of a ridge cover of $C^4$ such that
$A_i$ contains only ridges which belong to a single facet. Then some
other $A_j$ in the ridge cover of $C^4$ will contain a pair of
antipodal edges.
\end{lemma}

\begin{proof}
Without loss of generality, let $A_i$ contain the ridges of the bottom
facet ({\tt XXX0}). Then $A_i$ does not cover any of the twelve
spanning ridges (those whose coordinates end in {\tt X}). The spanning
ridges must therefore be covered by the other three sets.

These twelve spanning ridges of $C^4$ correspond naturally to the
edges of $C^3$. Namely, if the {\tt X} at the end of one of the
ridge's coordinates in $C^4$ is removed, we simply have the
coordinates of an edge in $C^3$. For example, {\tt X00X} lies in $C^4$
and {\tt X00} lies in $C^3$.  

Therefore, a cover of the ridges of $C^4$ by three sets corresponds to
a cover of the edges of $C^3$ by three sets, which by Lemma \ref{weak}
implies there is a set of the cover that contains a pair of 
antipodal vertices.  By concatenating {\tt X} to the end of their
coordinate representations, these vertices correspond to edges in
$C^4$ that are antipodal and contained in one of the original sets of
the ridge cover of $C^4$.
\end{proof}

We now prove Theorem \ref{strong} on $C^4$.

\begin{proof}
Lemmas \ref{6ridgesperset} and \ref{nofacet4} show that if four sets of a ridge cover of $C^4$ each contain no antipodal edges, they must all be asterisk sets.  
Furthermore, no two asterisk sets of a $4$-set cover may contain a common ridge because each set has six ridges and $C^4$ contains 24 ridges.

Denote the four asterisk vertices by $v_1, v_2, v_3$, and $v_4$.
If the asterisk sets of $v_i$ and $v_j$ do not contain the same ridge, they must differ in at least $3$ or $4$ coordinates; otherwise any two coordinates on which they agree would determine a spanning ridge that lies in both asterisk sets.  This observation shows that $v_1$ must differ from each of $v_2$ and $v_3$ in at least $3$ coordinates; hence $v_2$ and $v_3$ must agree on at least $3+3-4=2$ coordinates, and therefore there is a spanning ridge covered by $Ast(v_2)$ and $Ast(v_3)$.
So four (indeed, three) asterisk sets cannot cover ridges of $C^4$ without overlap, and therefore four asterisk sets cannot cover the ridges of $C^4$, contradicting our assumption that a $4$-set ridge cover of $C^4$ was possible with no set containing antipodal edges.

Therefore some set of a $4$-set ridge cover of $C^4$ must contain antipodal edges.
\end{proof}

\section{The sharpness of $k(d)$}
\label{sec:sharpness}

We have so far shown that a $d$-set ridge cover of $C^d$ contains a set that is $k(d)$-self-antipodal, where
$$
k(d)=
\left\{
\begin{array}{ll}
d-2, & d=1 \\
d-3, & 2 \leq d \leq 4 \\
d-4, & d \geq 5. \\
\end{array}
\right.
$$
This proves all aspects of Theorem \ref{complete} except the sharpness of $k(d)$, which is the statement of Theorem \ref{sharpness}.

To prove that these values of $k(d)$ are sharp for all $d\neq 5$, we show
that there exists a $d$-set ridge cover of $C^d$ that contains a pair of antipodal $k(d)$-faces but not antipodal $(k(d)+1)$-faces.  
For $d=5$ we show that there there exists a $d$-set ridge cover that does not contain any antipodal $(d-2)$-faces, i.e., ridges.
We first address the cases for which $1 \leq d \leq 5$, which are each proven by explicit examples.

On $C^1$, ridges are empty sets, and an empty set is antipodal to itself and by convention $(-1)$-dimensional.

For $C^2$, $C^3$, $C^4$, and $C^5$, we exhibit a $d$-set ridge cover that does not contain antipodal ridges. In fact, for all these cubes, we exhibit a 2-set ridge cover $A_1, A_2$ that does not contain antipodal ridges: since ridges come in antipodal pairs, for any pair of antipodal ridges, put one in $A_1$ and the other in $A_2$. Then, neither set contains antipodal ridges.

For $d \geq 6$, we construct a $d$-set ridge cover that is not
$(d-3)$-self-antipodal.  Because $d \geq 6$ if and only if $4 + \left
\lceil \frac{d}{3} \right \rceil \leq d$, the sharpness of $k(d)$ for
$d \geq 6$ results from the following lemma.

\begin{lemma}\label{upperbound}
Consider a $d$-cube (with $d \geq 4$).  There exists a collection of $4 + \left \lceil \frac{d}{3} \right \rceil$ sets that covers the ridges of $C^d$ such that no set contains antipodal peaks.
\end{lemma}

\begin{proof}
We give a construction of such a cover, which
consists of four asterisk sets together with $\left \lceil \frac{d}{3}
\right \rceil$ more sets which cover the remaining ridges.

Let us start by choosing four asterisk vertices.  Since our goal is to cover $C^d$ with as few sets as possible, we want to choose four vertices such that the union of their asterisk sets covers as many ridges as possible, or equivalently, such that the number of ridges that are covered by more than one set is minimized.

Intuitively, if we choose two vertices that are close to each other, there will be many ridges that intersect both vertices.  Hence, these ridges belong to the asterisk sets of both vertices.  Therefore, we want to choose four vertices such that the ``distance'' between each pair of vertices is maximized.

Using the coordinate notation, we can measure the ``distance'' between two vertices by using Hamming distance, the number of common coordinates that they share.  For example, in $d = 8$, the vertex $a_1 = \tt 000 \ 000 \ 00$ is ``closer'' to $a_2 = \tt 000 \ 000 \ 01$ than it is to $a_3 = \tt 010 \ 001 \ 11$.  Note that $a_1$ and $a_2$ have $d-1$ coordinates in common, which means that they lie on the same edge.  On the other hand, $a_1$ and $a_3$ have only $d - 4$ coordinates in common, which means that they lie on the same $4$-face.  We can easily check that the intersection of $Ast(a_1)$ and $Ast(a_3)$ contains fewer ridges than the intersection of $Ast(a_1)$ and $Ast(a_2)$.

We shall construct asterisk vertices using strings of $\tt 0$'s and $\tt 1$'s in
various lengths, called {\em blocks}.
For each $d$, define the blocks $\tt \alpha$, $\tt \beta$, and $\tt
\gamma$ to be strings containing all $\tt 0$'s whose lengths are specified
by the following table:
$$\begin{array}{|c|c|c|c|} 
\hline
& & & \\
d \mod 3 & |\alpha| & |\beta| & |\gamma| \\
& & & \\
\hline
& & & \\
0 & \frac{d}{3} & \frac{d}{3} & \frac{d}{3} \\
& & & \\
1 & \left \lceil \frac{d}{3} \right \rceil & \left \lfloor \frac{d}{3} \right \rfloor & \left \lfloor \frac{d}{3} \right \rfloor \\
& & & \\
2 & \left \lceil \frac{d}{3} \right \rceil & \left \lceil \frac{d}{3} \right \rceil & \left \lfloor \frac{d}{3} \right \rfloor\\
& & & \\
\hline
\end{array}$$
Let the block $\tt \bar{\alpha}$ be an $|{\tt \alpha}|$-tuple
containing all $\tt 1$'s, and define $\tt \bar{\beta}$ and $\tt
\bar{\gamma}$ similarly.

Without loss of generality, we choose four asterisk vertices by
concatenating blocks as follows:  
let $v_1 =  \alpha \beta \gamma$, 
$v_2 =  \bar{\alpha} \bar{\beta} \gamma$, 
$v_3 =  \alpha       \bar{\beta} \bar{\gamma}$, and 
$v_4 =  \bar{\alpha} \beta       \bar{\gamma}$.

Thus each vertex consists of three blocks, and for any vertex $v_i$, each of the other three vertices share common coordinates with $v_i$ on precisely one of the blocks.
For example, when $d = 8$, the four asterisk vertices are $v_1 = \tt 000 \ 000 \ 00$, $v_2 = \tt 111 \ 111 \ 00$, $v_3 = \tt 000 \ 111 \ 11$, and $v_4 = \tt 111 \ 000 \ 11$.  The vertex $v_1$ has the same first block as $v_3$, the same second block as $v_4$, and the same third block as $v_2$.

Let the first four cover sets be $A_i = Ast(v_i)$, for $i = 1, \ldots, 4$, where $Ast(v_i)$ denotes the asterisk of $v_i$.  We would like to count the number of ridges that are not covered by these four sets.  

\begin{lemma} \label{totalnotcovered}
In the choice of $A_i$ described above, each ridge is covered by at most two of the four asterisk sets, where the number of ridges that are doubly-covered is:
\begin{eqnarray*}
\frac{1}{3} d (d-3) & for \ d \equiv 0 \pmod{3} \\
\frac{1}{3} (d-1)(d-2) & for \ d \equiv 1,2 \pmod{3}
\end{eqnarray*}
Moreover, this is equal to the number of ridges that are not covered by the four asterisk sets.
\end{lemma}

\begin{proof}
We first show that each ridge is covered by at most two asterisk sets.
Consider a ridge $R$ that belongs to at least two asterisk sets, say
$Ast(v_i)$ and $Ast(v_j)$.  Let $v_i$ and $v_j$ have the $k$-th block
of coordinates in common.  Hence, $R$ must have its fixed coordinates
in the $k$-th block.  Since neither $v_i$ nor $v_j$ has the same
$k$-th block as either of the other two vertices (true by our choice
of the four asterisk vertices), $R$ does not belong to either of the
remaining two asterisk sets.  Therefore, any ridge belongs to at most
two asterisk sets.

Observe that a ridge belongs to an asterisk set if its two fixed
coordinates agree with the corresponding coordinates of the asterisk
vertex. Thus, a ridge is covered by both $Ast(v_i)$ and $Ast(v_j)$
only if its two fixed coordinates agree with the corresponding
coordinates of both $v_i$ and $v_j$.  Note that if $v_i$ and $v_j$
have exactly $k$ coordinates in common, there are ${k \choose 2}$ such
ridges.

For all $i \neq j$, let $k_{i, j} = k_{j, i}$ be the number of common
coordinates between vertices $v_i$ and $v_j$.  Then the number of
ridges that are doubly-covered by the four asterisk sets is:
\begin{equation} \label{doublycovered} 
\displaystyle\sum_{i < j}{k_{i,j} \choose 2} = {k_{1,2} \choose 2} + {k_{1,3} \choose 2} + {k_{1,4} \choose 2} + {k_{2,3} \choose 2} + {k_{2,4} \choose 2} + {k_{3, 4} \choose 2}.
\end{equation}
By our choice of vertices, we have:
\begin{eqnarray*}
k_{1,3} = k_{2,4} &=& |{\tt \alpha}| = \left \lceil \frac{d}{3} \right \rceil, \\
k_{1,4} = k_{2,3} &=& |{\tt \beta}| = \left \{ \begin{array}{l l}\left \lfloor \frac{d}{3} \right \rfloor & d \equiv 0, 1 \pmod{3} \\ \left \lceil \frac{d}{3} \right \rceil & d \equiv 2 \pmod{3} \end{array} \right. , \\
k_{1,2} = k_{3,4} &=& |{\tt \gamma}| = \left \lfloor \frac{d}{3} \right \rfloor.
\end{eqnarray*}

Substituting these values into (\ref{doublycovered}), the number of ridges that are doubly-covered is
\begin{eqnarray*}
\frac{1}{3} d (d-3) & for \ d \equiv 0 \pmod{3} \\
\frac{1}{3} (d-1)(d-2) & for \ d \equiv 1,2 \pmod{3}
\end{eqnarray*}

Since an asterisk set contains ${d \choose 2}$ ridges and $C^d$ contains $4{d \choose 2}$ ridges, and since a ridge that is covered by the union of the asterisk sets is covered by at most two of them, the number above is equal to the number of ridges that are not yet covered by the asterisk sets.
\end{proof}

To describe the coordinate notations of the ridges that are not yet covered, we define the following sets.  Let $A'$ be the set of $|{\tt \alpha}|$-tuples such that each tuple has only two fixed coordinates (and $\tt X$'s in the remaining coordinates) where one of the fixed coordinates is ${\tt 0}$ and the other fixed coordinate is ${\tt 1}$. Define $B'$ and $\Gamma'$ similarly.  Then, the ridges that are not yet covered are of the following forms:
\begin{eqnarray}
\tt \alpha' & \tt X \ldots X & \tt X \ldots X \label{form1}\\
\tt X \ldots X & \tt \beta' & \tt X \ldots X \label{form2}\\
\tt X \ldots X & \tt X \ldots X & \tt \gamma' \label{form3}
\end{eqnarray}
where ${\tt \alpha'} \in A'$, ${\tt \beta'} \in B'$, and ${\tt \gamma'} \in \Gamma'$.

We need to partition these ridges (that are not yet covered) into sets
such that each set does not contain any pair of antipodal peaks.  We
claim that we can do so in $\left \lceil \frac{d}{3} \right \rceil$
sets.  To do this, consider ridges of each of the forms (\ref{form1}),
(\ref{form2}), and (\ref{form3}) separately.  Let us first consider
the ridges of form (\ref{form1}).  There are $2{|{\tt \alpha}| \choose
  2}$ many of them, but we can partition them into $|{\tt \alpha}|$
sets by putting into the same set the ridges whose $\tt 0$ appears at
the same given coordinate.

\begin{example} \label{exampleridges}
Consider $d = 8$.  The ridges of form (\ref{form1}) that are not yet covered are
$$
\begin{array}{c c}
\tt X10 \ XXX \ XX & \tt 10X \ XXX \ XX \\
\tt 1X0 \ XXX \ XX & \tt 0X1 \ XXX \ XX \\
\tt O1X \ XXX \ XX & \tt X01 \ XXX \ XX \\
\end{array}
$$
Then, we can partition these ridges into three sets such that the location of the ${\tt 0}$ in each set is the same.
$$\begin{array}{|c|c|c|}
\hline
\tt 01X \ XXX \ XX & \tt 10X \ XXX \ XX & \tt 1X0 \ XXX \ XX \\
\tt 0X1 \ XXX \ XX & \tt X01 \ XXX \ XX & \tt X10 \ XXX \ XX \\
\hline
\end{array}$$
\end{example}

Observe that for two ridges to contain no pair of antipodal peaks, one of the following must hold.
\begin{enumerate}
	\item They have no fixed coordinates in the same location, or
	\item if at least one of the fixed coordinates is in the same location, the value of this coordinate must be the same in both ridges.
\end{enumerate}

Since the second condition is fulfilled by each of the $| {\tt\alpha}
|$ sets above, then none of these sets contains antipodal peaks.
Moreover, there are $\left \lceil \frac{d}{3} \right \rceil - 1$
ridges with $\tt 0$ at each coordinate of the first block.  Since
there are $\left \lceil \frac{d}{3} \right \rceil$ possible places for
$\tt 0$ in the first block, we can form $\left \lceil \frac{d}{3}
\right \rceil$ sets, containing a total of $\left \lceil \frac{d}{3}
\right \rceil (\left \lceil \frac{d}{3} \right \rceil - 1)$ ridges,
which is the total number of ridges of the form (\ref{form1}).

We partition the ridges of form (\ref{form2}) and the ridges of form (\ref{form3}) in similar manners.  We note that the union of the sets of ridges of distinct forms will not contain antipodal peaks because any two ridges of distinct forms have no fixed coordinates in the same location.  Thus, the total number of sets needed to cover all the remaining ridges is $\max \{ |\alpha|, |\beta|, |\gamma| \}$, which is $\left \lceil \frac{d}{3} \right \rceil$.

\begin{example}
Again, consider $d = 8$.  We partition the ridges of form (\ref{form2}) and the ridges of form (\ref{form3}) as we did with ridges of form (\ref{form1}) in Example \ref{exampleridges}.  In this dimension, the ridges of forms (\ref{form1}) and (\ref{form2}) are each partitioned into three sets, while the ridges of form (\ref{form3}) are partitioned into two sets.  Hence, we can form three sets, each containing at most one partition of a given form.  In this example, we let the set $A_{4 + i}$ contain ridges whose fixed $\tt 0$ is located at the $i$-th coordinate of the block it is in.  Then, we obtain the following $\left \lceil \frac{8}{3} \right \rceil$ sets that cover the remaining ridges of $C^8$.
$$\begin{array}{|c|c|c|}
\hline
A_5 & A_6 & A_7\\
\hline
\tt 01X \ XXX \ XX & \tt 10X \ XXX \ XX & \tt 1X0 \ XXX \ XX \\
\tt 0X1 \ XXX \ XX & \tt X01 \ XXX \ XX & \tt X10 \ XXX \ XX \\
 & & \\
\tt XXX \ 01X \ XX & \tt XXX \ 10X \ XX & \tt XXX \ 1X0 \ XX \\
\tt XXX \ 0X1 \ XX & \tt XXX \ X01 \ XX & \tt XXX \ X10 \ XX \\
 & & \\
\tt XXX \ XXX \ 01 & \tt XXX \ XXX \ 10 & \\
\hline
\end{array}$$
\end{example}

Hence, for all $d \geq 4$, the ridges of the $d$-cube can be covered with $4 + \left \lceil \frac{d}{3} \right \rceil$ sets, where each set does not contain antipodal peaks. This completes the proof of Lemma \ref{upperbound}. 
\end{proof}

By Lemma \ref{upperbound}, for $d \geq 4$, there exists a collection
of $4 + \left \lceil \frac{d}{3} \right \rceil$ sets, each containing
no pair of antipodal peaks, that covers all the ridges of $C^d$.
Recall that for $d \geq 6$, $4 + \left \lceil \frac{d}{3} \right
\rceil \leq d$.  Thus, for all $d \geq 6$, we can construct a 
$d$-set ridge cover such that each 
set in the cover contains no antipodal peaks of
$C^d$.  This implies that $k(d)$ is sharp for all $d \geq 6$.

\section{Discussion}

In this paper, we have shown how the classical LSB theorem on covers
of spheres has a combinatorial analogue for faces of hypercubes.
In particular, since the boundary of a hypercube is a topological
sphere, we expect that if we cover the facets of a ($d$-dimensional) hypercube
by $d$ collections of
facets, then one collection should contain a pair of antipodal points;
but it is somewhat surprising that if we cover the ridges of a
hypercube by $d$ sets, each a union of ridges, then some set 
contains a pair of antipodal points.

In fact, we showed that some set contains antipodal
faces of higher dimension, and determined a maximum dimension of
face that can be guaranteed, sharp in all dimensions except $d=5$.
We now discuss some possible directions to extend our work.

We have considered some aspects of the relationship among four variables:
the dimension of antipodality, the number of sets in the cover, the
dimension of the faces covered, and the dimension of the hypercube. As
an extension, we might vary any three of these parameters and draw
conclusions about the fourth. For example, suppose that we cover the
peaks of $C^d$ such that no set contains antipodal vertices. We might then ask 
what is the minimum number of sets required for such a
cover? (When $d=3$ this number is two, and when $d=4$ this number is
at most four.)

Another extension that we suggest is to generalize our result to other
highly-symmetric polytopes such as the Platonic solids or certain
prisms. For example, consider a prism that is the product of an interval with a $(2m)$-gon.
Note that the
ridges of a 3-dimensional hexagonal prism can be covered by three
sets such that no set contains antipodal vertices. What, then, are the
necessary conditions to guarantee that some set in a ridge cover of a
polytope contains antipodal vertices (or antipodal faces of higher
dimensions)? 

A final extension is to consider a subdivision of the $d$-cube's facets into identical smaller $(d-1)$-cubes using $(d-2)$-cubes. For example, we might subdivide the 2-faces of the 3-cube into $n^2$ smaller squares. If we then let $d$ sets cover the ridges on the subdivided facets of the $d$-cube, is it true that some set must contain antipodal vertices?  We have found that it is true for $d=3$ and $n=2$. If this result were true as $n \rightarrow \infty$, it might be used to give a combinatorial proof of the classical LSB theorem.

\end{document}